\documentclass[a4paper,12pt]{article}

\usepackage{amsthm}
\usepackage{mathrsfs}
\usepackage{amsfonts,amsmath,oldgerm,amssymb,amscd,tikz-cd}
\usepackage{pdfpages}
\usepackage{epsfig}

\usepackage{euscript}
\usepackage{amsthm}
\usepackage{theoremref}
\usepackage[all]{xy}
\DeclareMathAlphabet{\mathpzc}{OT1}{pzc}{m}{it}

\newtheorem{thm}{Theorem}[section]

\newtheorem{lem}[thm]{Lemma}
\newtheorem{prop}[thm]{Proposition} 
\newtheorem{cor}[thm]{Corollary}
\newtheorem{rem}[thm]{Remark}
\newtheorem{ex}[thm]{Example}

\newcommand{\p}{\mathpzc{p}}

\newcommand{\bC}{\mathbb C}

\newcommand{\A}{\mathbb A}

\newcommand{\td}{\operatorname{tr.deg}}

\newcommand{\dk}{\operatorname{DK}}
\newcommand{\ml}{\operatorname{ML}}

\title{On Generalised Danielewski and Asanuma varieties }
\author{Parnashree Ghosh$^*$ and Neena Gupta$^{**}$\\
	{\small{\it  Stat-Math Unit, Indian Statistical Institute,}}\\ 
	{\small{\it 203 B.T.Road, Kolkata-700108, India}}\\
	{\small{\it e-mail : $^*$parnashree$\_$r@isical.ac.in, ghoshparnashree@gmail.com}}\\
	{\small{\it e-mail : $^{**}$neenag@isical.ac.in, rnanina@gmail.com}}
}

\begin{document}
	\date{}
	\maketitle
	\abstract{In this paper we extend a result of Dubouloz on the Cancellation Problem in higher dimensions ($\geqslant 2$) over the field of complex numbers to fields of arbitrary characteristic.
We then apply the generalised result to describe the Makar-Limanov and Derksen invariant of 
 generalised Asanuma varieties under certain hypotheses. We also establish a necessary and sufficient condition for certain generalised Asanuma varieties to be isomorphic to polynomial rings.
 }
	
	\smallskip
	
	
	
	\section{Introduction}
	Throughout this paper, all rings will be commutative with unity, $k$ will denote a field and for a commutative ring $R$, $R^{[n]}$ will denote a polynomial ring in $n$ variables over $R$. $R^{*}$ will denote the group of units in $R$.
	
	The Cancellation Problem, a fundamental problem in the area of Affine Algebraic Geometry, asks the following (cf. \cite{eh}):
	
	\medskip
	\noindent
	{\bf Question 1.} Let $D$ and $E$ be two affine domains over a field $k$ such that $D^{[1]} =_k E^{[1]}$. Does this imply $D \cong_k E$?
	
	\medskip

    If $E=k^{[n]}$, then the above question is known as the Zariski Cancellation Problem. 
      The answer to Question 1 is affirmative for one dimensional affine domains (cf. \cite{aeh}).
      However, there are counterexamples in dimensions greater than or equal to two. 
      In \cite{dan}, Danielewski constructed a family of two dimensional pairwise non-isomorphic smooth complex varieties which are counterexamples to the Cancellation Problem. 
      In \cite{cra1}, A. J. Crachiola extended Danielewski's examples over arbitrary characteristic.
      In \cite{dub}, Dubouloz constructed higher dimensional ($\geqslant 3$) analogues of the Danielewski varieties, which are counterexamples to this problem. More precisely,
for $\underline{r}:=(r_1,\ldots,r_m) \in \mathbb{Z}^m_{\geqslant 1}$ and $F\in k^{[m+1]}$,      
       he studied the affine varieties, defined by the following integral domains when $k=\mathbb{C}$: 
      $$
      B_{\underline{r}}:= B(r_1,\ldots,r_m,F) = \frac{k[T_1,\ldots,T_m,U,V]}{(T_1^{r_1}\cdots T_m^{r_m}U-F(T_1,\ldots,T_m,V))},
      $$
      where $F(T_1, \dots, T_m, V)$ is monic in $V$ and $d:= \deg_V F>1$. Note that
      setting $P(V):=F(0,\ldots,0,V)$, we have $d=\deg_V P= \deg_{V} F=d( >1)$. In
      the above setting, Dubouloz proved (\cite{dub}, Corollary 1.1, Corollary 1.2):
      
      \medskip
      \noindent
      {\bf Theorem A.}
   Suppose $F=\prod_{i=1}^{d} (V-\sigma_i(T_1,\ldots,T_m))$, 
    where $\sigma_i$'s are of the form 
        $$
        \sigma_i(T_1,\ldots,T_m)=a_i+T_1\cdots T_mf_i(T_1,\ldots,T_m),~~
         f_i \in \mathbb{C}^{[m]} \text{~for all~} i, ~~1\le i\le d,
         $$
and $a_i \in \bC$ are such that $a_i \neq a_j$ for $i \neq j$. 
Then the following statements hold:
\begin{enumerate}
\item[\rm(i)] Suppose that either of the following conditions hold:\\
(a)  $\underline{r} \in \mathbb{Z}^m_{>1}$ and $\underline{s} \in \mathbb{Z}^m_{\geqslant 1} \setminus \mathbb{Z}^m_{>1}$\\
(b) $\underline{r}, \underline{s} \in \mathbb{Z}^m_{>1}$ are such that the sets $\{r_1,\ldots,r_m\}$ and $\{s_1,\ldots,s_m\}$ are distinct.\\
Then $B_{\underline{r}} \ncong_k B_{\underline{s}}$.
\item[\rm(i)] $B_{\underline{r}}^{[1]} \cong_k B_{\underline{s}}^{[1]}$ for any 
$\underline{r}, \underline{s} \in \mathbb{Z}^m_{\geqslant 1}$. 
\end{enumerate}

	  In this paper, we show that the result of Dubouloz extends to fields of arbitrary characteristic. More generally, we establish the following (\thref{isod}, \thref{stiso}):
	
	\medskip
	\noindent
	{\bf Theorem B.}
	Let $F, F^{\prime} \in k[T_1,\ldots,T_m,V]$ be such that
		$$
	F= a_0+a_1 V+\cdots + a_{d-1}V^{d-1} + V^d,~~\text{and}
	$$
	$$
	F^{\prime}= a_1+ 2a_2V+ \cdots + (d-1)a_{d-1} V^{d-2} + d V^{d-1}
	$$
	for some $a_i \in k[T_1,\ldots,T_m]$, $0 \leqslant i \leqslant d-1$.
	Then the following statements hold:
\begin{enumerate}
\item[\rm(i)] Suppose that either of the following conditions hold:\\
(a)  $\underline{r} \in \mathbb{Z}^m_{>1}$ and $\underline{s} \in \mathbb{Z}^m_{\geqslant 1} \setminus \mathbb{Z}^m_{>1}$\\
(b) $\underline{r}, \underline{s} \in \mathbb{Z}^m_{>1}$ are such that $\underline{r}=(r_1,\ldots,r_m)$ and $\underline{s}=(s_1,\ldots,s_m)$ are not permutation of each other.\\
Then $B_{\underline{r}} \ncong_k B_{\underline{s}}$.

\item[\rm(ii)] $B_{\underline{r}}^{[1]} \cong_k B_{\underline{s}}^{[1]}$ for any 
$\underline{r}, \underline{s} \in \mathbb{Z}^m_{\geqslant 1}$, whenever $\left(F,F^{\prime}\right)=k[T_1,\ldots,T_m,V]$. 
\end{enumerate}	 
Now consider the family of affine domains defined as follows:
	\begin{equation}\label{A}
		A= \frac{k\left[ X_{1}, \ldots , X_{m}, Y,Z,T\right]}{(X_{1}^{r_{1}} \cdots X_{m}^{r_{m}}Y -H(X_{1}, \ldots , X_{m}, Z,T))},\,\ r_i > 1 \text{~for all~} i, 1 \leqslant i \leqslant m,
	\end{equation}
	where  $h(Z,T):=H(0, \ldots, 0, Z,T) \neq ~0$. 
	We shall call a variety defined by a ring of  type \eqref{A} as ``Generalised Asanuma variety" since
for $m=1$ and for a field $k$ of positive characteristic,
and $H=Z^{p^e}+T+T^{sp}$, where $e,s$ are positive integers such that 
$p^e \nmid sp$ and $sp \nmid p^e$,
 rings of the above type were constructed by T. Asanuma	in \cite{asa87}, as an illustration of non-trivial $\A^2$-fibrations (defined in section 2) over a PID not containing $\mathbb{Q}$. 
Special cases of the rings of above type have played crucial roles in solutions
	to central problems on affine spaces. The Makar-Limanov and Derksen invariants (defined in section 2) 
were the key tools in some of the results.
The case $m=1$ accommodates the  famous Russell-Koras threefold $x^2y+x+z^2+t^3=0$
over $k= \bC$ which arose in the context of the Linearisation Problem for ${\bC}^{[3]}$ 
and which is now a potential candidate for a counterexample to Zariski Cancellation problem (ZCP) for the affine three space in characteristic zero
(see \cite{Gicm} for details). 

Recently, the authors have given several necessary and sufficient conditions for the ring $A$ to be a polynomial ring in $m+2$ variables over $k$ for a special form of $H$ (\cite[Theorem 3.10]{asa}).
In section 4, we establish a new criterion (\thref{thpoly}): 

\medskip
\noindent
{\bf Theorem C.}
Let $A$ be the affine domain as in \eqref{A}, with $H=h(Z,T)+(X_1 \cdots X_m)g$ for some $g \in k[X_1,\ldots,X_m,Z,T]$.
 Then the following statements are equivalent:
\begin{itemize}
	\item [\rm (i)] $A$ is geometrically factorial over $k$ and there exists an exponential map $\phi$ on $A$ satisfying 
	$k[x_1,\ldots, x_m] \subseteq A^{\phi} \not \subseteq k[x_1,\ldots, x_m,z,t]$.
	
	\item [\rm (ii)] $A=k^{[m+2]}$.
\end{itemize}

\medskip
 
For the affine domain $A$ in \eqref{A}, when $k$ is a field of positive characteristic, and $H=h(Z,T)$, where
$h$ is a nontrivial line in $k[Z,T]$ (defined in section 2), the second author had shown (\cite[Corollary 3.9]{inv}, \cite[Corollary 3.8]{adv})
that for each $m \geqslant 1$, the corresponding ring $A$ is a counterexample to the ZCP
in dimension $(m+2)$, i.e., $A^{[1]}=k^{[m+3]}$ but $A \neq k^{[m+2]}$. A crucial step in the proof was a result on
the Derksen invariant quoted as \thref{dk} of the present paper.  In section 4, we address the following converse of \thref{dk}:

\medskip
\noindent
{\bf Question 2.} Let $A$  be as in \eqref{A}. Suppose that there exists a system of coordinates $\{Z_1, T_1\}$ of $k[Z,T]$ such that 
$h(Z,T)= a_0(Z_1)+a_1(Z_1)T_1$. Is the Derksen invariant of $A$ equal to $A$? 
 
 \medskip
 
 We shall apply Theorem B to give an affirmative answer to the above question when $H$ is of a certain form (\thref{ml1}).
We also give a complete description of the Derksen and Makar-Limanov  invariants for
	 $$
	 A = \frac{k[X_{1},\ldots, X_{m},Y,Z,T]}{\left( X_{1}^{r_{1}}\cdots X_{m}^{r_{m}}Y-h(Z,T)\right)},
	 $$
	 when $k$ is an infinite field and $A$ is a regular domain (\thref{ml2}).
 
 	In the next section, we recall some definitions, well-known results, and basic properties of exponential maps on a $k$-algebra which will be used in this paper.
	
	\section{Preliminaries} 
 Let $R$ be a ring and $B$ an $R$-algebra. For a prime ideal $\p$ of $R$, 
let $k(\p)$ denote the field $\frac{R_{\p}}{\p R_{\p}}$. 
 Capital letters like $X,Y,Z,T,U,V,W$, $X_1,\ldots,X_m, T_1,\ldots,T_m$, etc., will denote indeterminates over respective ground rings or fields.  
 
 We first recall a few definitions.

\medskip
\noindent 
{\bf Definition.} A polynomial $h \in k[X,Y]$ is said to be a {\it line} in $k[X,Y]$ if $\frac{k[X,Y]}{(h)} = k^{[1]}$. Furthermore, if $k[X,Y] \neq k[h]^{[1]}$, then such an $h$ is said to be a {\it non-trivial line} in $k[X,Y]$. 

\medskip
\noindent
{\bf Definition.} A finitely generated flat $R$-algebra $B$ is said to be an {\it $\A^n$-fibration} over $R$
if $B \otimes_R k(\p) = k(\p)^{[n]}$ for every prime ideal $\p$ of $R$.

\medskip
\noindent
{\bf Definition.} A $k$-algebra $B$ is said to be {\it geometrically factorial} over $k$ if for every algebraic field extension $L$ of $k$, $B \otimes_k L$ is a UFD.  

\medskip
 We now record a special case of a result on the triviality of separable $\mathbb{A}^{1}$-forms \cite[Theorem 7]{dutta}.

\begin{lem}\thlabel{sepco}
	Let $h \in k[Z,T]$ be such that $L[Z,T]=L[h]^{[1]}$, for some separable field extension of $L$ of $k$. Then $k[Z,T]=k[h]^{[1]}$. 
\end{lem}

Next we recall the concept of an exponential map on a $k$-algebra and two invariants on the $k$-algebra.

	\medskip
	\noindent
	{\bf Definition.}
		Let $B$ be a $k$-algebra and $\phi: B \rightarrow B^{[1]}$ be a $k$-algebra homomorphism. For an indeterminate $U$ over $B$, let $\phi_{U}$ denote the map $\phi: B \rightarrow B[U]$. Then $\phi$ is said to be an {\it exponential map} on $B$, if the following conditions are satisfied:
		
		\begin{itemize}
			\item [\rm(i)]  $\epsilon_{0} \phi_{U} =id_{B}$, where $\epsilon_{0}: B[U] \rightarrow B$ is the evaluation map at $U=0$.
			
			\item [\rm (ii)]  $\phi_{V} \phi_{U}=\phi_{U+V}$, where $\phi_{V}:B \rightarrow B[V]$ is extended to a $k$-algebra homomorphism $\phi_{V}:B[U] \rightarrow B[U,V]$, by setting $\phi_{V}(U)=U$.
		\end{itemize}

	The {\it ring of invariants} of $\phi$ is a subring of $B$ defined as follows:
	$$
	B^{\phi}:=\left\{b \in B \, | \phi(b)=b  \right\}.
	$$
	If $B^{\phi} \neq B$, then $\phi$ is said to be a non-trivial exponential map. 
	
	Let EXP($B$) denote the set of all exponential maps on $B$. The {\it Makar-Limanov invariant} of $B$ is defined by
	
	$$
	\ml(B):= \bigcap_{\phi \in \text{EXP}(B)} B^{\phi}
	$$
	and the {\it Derksen invariant} is a subring of $B$ defined as
	
	$$
	\dk(B):=k\left[ b \in B^{\phi} \,|\, \phi \in \text{EXP}(B)\, \text{and}\, \, B^{\phi} \subsetneq B   \right].
	$$
	
	We now record some properties of exponential maps. For details one can refer to \cite{cra}, \cite[Chapter I]{miya}. 
	
	\begin{lem}\thlabel{prop}
		Let $B$ be an affine domain over a field $k$ and $\phi$ be a non-trivial exponential map on $B$. Then the following statements hold:
		
		\begin{itemize}
			\item [\rm(i)] $B^{\phi}$ is a factorially closed subring of $B$, i.e., for any non-zero $a,b \in B$, if $ab \in B^{\phi}$, then $a, b \in B^{\phi}$. In particular, $B^{\phi}$ is algebraically closed in $B$.

			\item[\rm(ii)]  Let $x \in B \setminus B^{\phi}$ be an element such that $\deg_U \phi_{U}(x)$ is of minimal positive value. Let $c$ be the leading coefficient of $U$ in $\phi_{U}(x)$. Then $B[c^{-1}]=B^{\phi}[c^{-1}][x]=(B^{\phi}[c^{-1}])^{[1]}$.
			In particular, $\td_k B^{\phi} = \td_k B-1$.
			
			\item[\rm(iii)]  Let $S$ be a multiplicatively closed subset of $B^{\phi} \setminus \{0\}$. Then $\phi$ will induce a non-trivial exponential map $S^{-1}\phi$ on $S^{-1}B$ such that $(S^{-1}B)^{S^{-1}\phi}=S^{-1}(B^{\phi})$.
		\end{itemize}

	\end{lem}

		 Next we recall the definition of a {\it rigid $k$-domain}.
	
	\medskip
	\noindent
	{\bf Definition.} A $k$-domain $D$ is said to be {\it rigid} if there does not exist any non-trivial exponential map on $D$. 
	
	\medskip
	
	For convenience, we record below an easy lemma.
	
	\begin{lem}\thlabel{rdk}
		Let $D$ be a $k$-domain which is not rigid. Then $\dk(D[W])=D[W]$.
	\end{lem}
	\begin{proof}
		Consider the exponential map 
		$\phi: D[W] \rightarrow D[W,U]$ defined by
		$$
		\phi(a)=a \text{~for all~} a \in D,\text{~and~} \phi(W)=W+U.
		$$
		It is easy to see that $(D[W])^{\phi}=D$. Again as $D$ is not rigid, we have a non-trivial exponential map $\psi$ on $D$. We extend $\psi$ to an exponential map $\widetilde{\psi}$ on $D[W]$ such that $\widetilde{\psi}|_{D}=\psi$ and $\widetilde{\psi}(W)=W$. Then $W \in (D[W])^{\widetilde{\psi}}$. Therefore, it follows that $\dk(D[W])=D[W]$.
	\end{proof}

	Next we recall the definition of {\it proper} and {\it admissible $\mathbb{Z}$-filtration} on an affine domain. 
	
	\medskip
	\noindent
	{\bf Definition.}
		Let $k$ be a field and $B$ be an affine $k$-domain. A collection $\{B_{n}\,\,| \, n \in \mathbb{Z}\}$ of $k$-subspaces of $B$ is said to be a {\it proper $\mathbb{Z}$-filtration} if
		
		\begin{itemize}
			\item [\rm(i)] $B_{n} \subseteq B_{n+1}$ for every $n \in \mathbb{Z}$.
			
			\item[\rm(ii)]  $B= \bigcup_{n \in \mathbb{Z}} B_{n}$.
			
			\item[\rm(iii)] $\bigcap_{n \in \mathbb{Z}} B_{n}=\{0\}$.
			
			\item[\rm(iv)]  $(B_{n} \setminus B_{n-1}).(B_{m}\setminus B_{m-1}) \subseteq B_{m+n} \setminus B_{m+n-1}$ for all $m,n \in \mathbb{Z}$.
		\end{itemize}
		
	For a non-zero element $b \in B$, the {\it degree} of $b$ with respect to a proper $\mathbb{Z}$-filtration $\{B_n\}_{n \in \mathbb{Z}}$ is defined	to be the smallest integer $d$ such that $b \in B_{d}$.
	A proper $\mathbb{Z}$-filtration $\{B_{n}\}_{n \in \mathbb{Z}}$ of $B$ is said to be {\it admissible} if there is a finite generating set $\Gamma$ of $B$ such that for each $n\in \mathbb{Z}$, every element in $B_{n}$ can be written as a finite sum of monomials from $k[\Gamma] \cap B_{n}$.
	
	\smallskip
	A proper $\mathbb{Z}$-filtration $\{B_{n}\}_{n \in \mathbb{Z}}$ of $B$ defines an associated graded domain defined by 
	$$
	gr(B):= \bigoplus_{n \in \mathbb{Z}} \frac{B_{n}}{B_{n-1}}.
	$$
	It also defines a natural map $\rho: B \rightarrow gr(B)$ such that $\rho(b)=b+B_{n-1}$, if $b \in B_{n} \setminus B_{n-1}$.
	
	The next theorem is on homogenization of exponential maps due to Derksen {\it{et\, al.}} \cite{dom}. The following version can be found in \cite[Theorem 2.6]{cra}.
	
	\begin{thm}\thlabel{dhm}
		Let $B$ be an affine domain over a field $k$ with an admissible proper $\mathbb{Z}$-filtration
		and $gr(B)$ be the induced $\mathbb{Z}$-graded domain. Let $\phi$ be a non-trivial exponential map on $B$. Then $\phi$
		induces a non-trivial homogeneous exponential map $\overline{\phi}$ on $gr(B)$ such that $\rho(B^{\phi}) \subseteq gr(B)^{\overline{\phi}}.$
	\end{thm}

In the rest of the section we will state some results on the integral domain $A$ defined by
 \begin{equation}\label{A1}
	A= \frac{k\left[ X_{1}, \ldots , X_{m}, Y,Z,T\right]}{(X_{1}^{r_{1}} \cdots X_{m}^{r_{m}}Y -H(X_{1}, \ldots , X_{m}, Z,T))},\,\ r_i > 1 \text{~for all~} i, 1 \leqslant i \leqslant m,
\end{equation}
where  $h(Z,T):=H(0, \ldots, 0, Z,T) \neq ~0$. 
Let $x_{1}, \ldots, x_{m},z,t$ denote the images of $X_{1}, \ldots , X_{m},Z,T$ respectively in $A$.
Recall that $A$ is the coordinate ring of a generalised Asanuma variety.

The following result is proved in \cite[Lemma 3.3]{adv}.

\begin{lem} \thlabel{subdk}
	$k[x_{1}, \ldots , x_{m}, z,t] \subseteq ~ \dk(A)$. 
\end{lem}
The second author proved the following proposition in \cite[Proposition 3.4(i)]{adv} where the result was stated with the hypothesis $\dk(A)=A$. 
However, the proof shows that the hypothesis can be relaxed to $k[x_{1}, \ldots , x_{m},z,t] \subsetneq \dk(A)$. We quote the result with the modified hypothesis. 
\begin{prop}\thlabel{dk}
	Suppose that $k$ is infinite and $k[x_{1}, \ldots , x_{m},z,t] \subsetneq \dk(A)$. 
	%If $m=1$ and $k$ be any field or $m>1$ and $k$ be an infinite field, 
	Then there exist $Z_1, T_1 \in k[Z,T]$ and $a_0, a_1 \in k^{[1]}$ such that $k[Z,T]=k[Z_1,T_1]$ and $h(Z,T)=a_0(Z_1)+a_1(Z_1)T_1$.
\end{prop}

The following result gives some necessary and sufficient conditions for a ring of type \eqref{A1} to be a UFD   (\cite[Proposition 3.6]{asa}).

\begin{prop}\thlabel{ufd2}

	Let $H(X_{1}, \ldots, X_{m}, Z,T)= h(Z,T)+ (X_{1} \cdots X_{m})g$, for some $g \in k[X_1,\ldots,X_m,Z,T]$. Then the following statements are equivalent:
	\begin{enumerate}
		
		\item[$\rm( i)$] $A$ is a UFD.

		\item[$\rm(ii)$] For each $j, 1\leqslant j \leqslant m$, either $x_{j}$ is prime in $A$ or  $x_j \in A^{*}$.
		
		\item[$\rm(iii)$]  $h(Z,T)$ is either an irreducible element in $k[Z,T]$ or $h(Z,T) \in k^{*}$.
	\end{enumerate}
\end{prop}

	Next we quote the following lemma from \cite[Lemma 3.1]{asa}.
	
	\begin{lem}\thlabel{rs}
		Let $E:=k[X_{1}, \ldots, X_{m}, T]$, $C= E[Z,Y]$, $g \in E[Z]$  and 
		$G=X_{1}^{r_{1}} \cdots X_{m}^{r_{m}}Y- X_{1} \cdots X_{m}g+Z \in C$.
		Then $C=E[G]^{[1]}$.
    \end{lem}

The next result gives a necessary and sufficient condition for a ring $A$ of type \eqref{A1} to be $k^{[m+2]}$ (\cite[Theorem 3.10]{asa}).

\begin{thm}\thlabel{equiv}
	Let $H(X_{1}, \ldots , X_{m},Z,T)=h(Z,T)+ (X_{1} \cdots X_{m})g$, for some $g \in k[X_1,\ldots,X_m,Z,T]$. Then the following statements are equivalent:
	\begin{enumerate}

		\item[\rm(i)] $A=k^{[m+2]}$.
		
		\item[\rm(ii)] $k[Z,T]=k[h(Z,T)]^{[1]}$.
		
	\end{enumerate}
\end{thm}

We also fix a notation which will be used in the paper. Let $C$ be a ring and $f \in C[W]$. Suppose $f=\sum_{i=0}^{n} c_i W^i$, for some $c_i \in C$, $0 \leqslant i \leqslant n$. Then $f_W$ will denote the polynomial 
$$
f_W:= \sum_{i=1}^{n} ic_i W^{i-1}.
$$

\section{Generalised Danielewski Varieties}	
 
 We first fix some notation which will be used throughout this section. For positive integers $r_1,\ldots,r_m$ and a polynomial $F=F(T_1,\ldots,T_m,V) \in k^{[m+1]}$ which is monic in $V$ with $\deg_{V} F >1$, $B(\underline{r},F)$ will denote the ring
$$
 B(\underline{r},F):=B(r_1,\ldots,r_m,F) = \frac{k[T_1,\ldots,T_m,U,V]}{(T_1^{r_1}\cdots T_m^{r_m}U-F(T_1,\ldots,T_m,V))},
$$
where $\underline{r}:=(r_1,\ldots,r_m) \in \mathbb{Z}^m_{\geqslant 1}$.
Set $P(V):=F(0,\ldots,0,V)$ and $d:= \deg_{V} P= \deg_{V} F (>1)$. Further, when $F$ is understood from the context, we will use the notation $B_{\underline{r}}$ to denote the above ring $B(\underline{r}, F)$, i.e., 
 \begin{equation}\label{b}
	B_{\underline{r}}:= B(r_1,\ldots,r_m,F) = \frac{k[T_1,\ldots,T_m,U,V]}{(T_1^{r_1}\cdots T_m^{r_m}U-F(T_1,\ldots,T_m,V))}.
\end{equation}
We call the varieties defined by these rings ``Generalised Danielewski varieties".
 Let $t_1,\ldots,t_m,u,v$ denote respectively the images of $T_1,\ldots,T_m,U,V$ in $B_{\underline{r}}$ and $R$ denote the subring $k[t_1,\ldots,t_m,v]$ of $B_{\underline{r}}$. 
	
Note that $R=k[t_1,\ldots,t_m,v] (=k^{[m+1]}) \hookrightarrow B_{\underline{r}} \hookrightarrow B_{\underline{r}}[t_1^{-1},\ldots,t_m^{-1}]=k[t_1^{\pm 1},\ldots,t_m^{\pm 1},v]$ and $F \in R$. Fix $(e_{1}, \ldots, e_{m}) \in \mathbb{Z}^{m}$. This $m$-tuple defines a proper $\mathbb{Z}$-filtration $\{B_n\}_{n \in \mathbb{Z}}$ on $B_{\underline{r}}$ as follows:

\noindent
 Set $C_n:= \bigoplus_{e_{1}i_{1}+\cdots+e_{m}i_{m}=n} k[v] t_1^{i_1} \ldots t_m^{i_m}$. Then the ring $k[t_1^{\pm 1},\ldots,t_m^{\pm 1},v]$ has the following $\mathbb{Z}$-graded structure (with $wt(t_i)=e_i$, $1 \leqslant i \leqslant m$):
$$
k[t_{1}^{\pm 1},\ldots,t_{m}^{\pm 1},v] = \bigoplus_{n \in \mathbb{Z}} C_{n}= \bigoplus_{n \in \mathbb{Z}, e_{1}i_{1}+\cdots+e_{m}i_{m}=n} k[v] t_{1}^{i_{1}} \cdots t_{m}^{i_{m}}.
$$
For every $n \in \mathbb{Z}$, set $B_{n} := \bigoplus_{i \leqslant n} C_{n} \cap B_{\underline{r}}$. 

Then $\{B_{n}\}_{n \in \mathbb{Z}}$ defines a proper $\mathbb{Z}$-filtration on $B_{\underline{r}}$ induced by $(e_1,\ldots,e_m)$ and for every $j$, $1 \leqslant j \leqslant m$, $t_{j} \in B_{e_{j}} \setminus B_{e_{j}-1}$.

 Let $e:= \deg (F)$ with respect to the given filtration. Then $u \in B_{\ell} \setminus B_{\ell-1}$, where $\ell=e-(r_{1}e_{1}+\cdots+r_{m}e_{m})$.
 Set 
 $$\Lambda:=\biggl\{ (\underline{i},j,q):=(i_1,\ldots,i_m,j,q) \in \mathbb{Z}^m_{\geqslant0} \times \mathbb{Z}_{>0} \times \mathbb{Z}_{\geqslant 0} \mid i_s <r_s \text{~for some~} s, 1\leqslant s \leqslant m \biggr\}.
 $$
Using the relation $t_1^{r_1}\cdots t_m^{r_m}u=F(t_1,\ldots,t_m,v)$, it can be noted that every element $b \in B_{\underline{r}}$ can be uniquely expressed as
\begin{equation}\label{b1}
	b= \widetilde{b}(t_1,\ldots,t_m,v)+ \sum_{(\underline{i},j,q) \in \Lambda} \alpha_{\underline{i}\, jq}\, t_{1}^{i_{1}} \cdots t_{m}^{i_{m}} u^j v^q ,
\end{equation}
where $\widetilde{b} \in R (=k[t_1,\ldots,t_m,v])$ and $\alpha_{\underline{i}\, jq} \in k^{*}$.

 Now since the filtration $\{B_n\}_{n \in \mathbb{Z}}$ is induced from the graded structure of the ring $B_{\underline{r}}[t_1^{-1},\ldots,t_m^{-1}]$, from the expression \eqref{b1} it follows that the filtration $\{B_n\}_{n \in \mathbb{Z}}$ is admissible with respect to the generating set $\Gamma=\{t_1,\ldots,t_m,u,v\}$ of $B_{\underline{r}}$ and the associated graded ring $gr(B_{\underline{r}})=\bigoplus_{n \in \mathbb{Z}} \frac{B_n}{B_{n-1}}$ is generated by the image of $\Gamma$ in $gr(B_{\underline{r}})$. 

The following lemma exhibits the structure of $gr(B_{\underline{r}})$.

\begin{lem}\thlabel{gr}
 
	 Let $F_e$ denote the highest degree homogeneous summand of $F$. If, for every $i$, $t_i \nmid F_e$ in $R=k[t_1,\ldots,t_m,v]$, then $$gr(B_{\underline{r}}) \cong \frac{k[T_1,\ldots,T_m,U,V]}{(T_1^{r_1}\cdots T_m^{r_m}U-F_e(T_1,\ldots,T_m,V))}.$$  
\end{lem}
\begin{proof}
Since $\deg(F)=e$, $t_1^{r_1} \cdots t_m^{r_m}u=F \in B_e \setminus B_{e-1}$ and hence, 
 \begin{equation}\label{F}
 	 \overline{t_1}^{r_1} \cdots \overline{t_m}^{r_m} \overline{u}=F_e(\overline{t_1},\ldots,\overline{t_m},\overline{v})
 \end{equation}
in $gr(B_{\underline{r}})$, where $\overline{t_1},\ldots,\overline{t_m},\overline{u},\overline{v}$ denote the images of $t_1,\ldots,t_m,u,v$ in $gr(B_{\underline{r}})$. Since $gr(B_{\underline{r}})$ can be identified with a subring of $gr(k[t_1^{\pm 1},\ldots,t_m^{\pm 1}, v]) \cong k[t_1^{\pm 1},\ldots,t_m^{\pm 1}, v]$, we have $\overline{t_1},\ldots,\overline{t_m},\overline{v}$ are algebraically independent in $gr(B_{\underline{r}})$ and hence $\dim gr(B_{\underline{r}})=m+1$.
	As $t_i \nmid F_e$ in $R$ for every $i$, 
	$
	\frac{k[T_1,\ldots,T_m,U,V]}{(T_1^{r_1}\cdots T_m^{r_m}U-F_e(T_1,\ldots,T_m,V))}
	$ 
	is an integral domain, and its dimension is $m+1$. Hence by \eqref{F}, we have the isomorphism:
	$$
	gr(B_{\underline{r}}) \cong \frac{k[T_1,\ldots,T_m,U,V]}{(T_1^{r_1}\cdots T_m^{r_m}U-F_e(T_1,\ldots,T_m,V))} .
	$$
\end{proof}

In \cite{dub}, Dubouloz showed that for $k=\mathbb{C}$, 
$\ml(B_{\underline{r}})=\mathbb{C}[t_1,\ldots,t_m]$, when $\underline{r} \in \mathbb{Z}^m_{>1}$. When $k$ is an algebraically closed field of characteristic zero and $F \in k[V]$, a similar result appears in \cite[Lemma 6.2]{Guff}. We extend this result over any field (of any characteristic) and arbitrary $F$. First we recall the following lemma (\cite{com}, Lemma 3.5).

\begin{lem}\thlabel{m1}
	Let $r_1>1$ and $q \in k[V]$ be such that $\deg_V q >1$. Then there is no non-trivial exponential map $\phi$ on $B(r_1,q) = k[T_1,U,V]/(T_1^{r_1}U-q(V))$ such that $u \in B(r_1,q)^{\phi}$, where $u$ denotes the image of $U$ in $B(r_1,q)$. 
\end{lem}

We now establish a generalisation of the above lemma. 

\begin{lem}\thlabel{m2}
	Let $\underline{r} \in \mathbb{Z}^m_{>1}$, and  $P(V) \in k[V]$ with $\deg_V P>1$. Then there is no non-trivial exponential map $\phi$ on $B(\underline{r},P)$, such that $u \in B(\underline{r},P)^{\phi}$. 
\end{lem}

\begin{proof}
	Let  
	$$
	B:=B(\underline{r},P) = \frac{k[T_1,\ldots,T_m,U,V]}{(T_1^{r_1}\cdots T_m^{r_m}U-P(V))}.
	$$
	We will prove the result by induction on $m$. For $m=1$, the result holds by \thref{m1}. Therefore we assume that $m\geqslant 2$. Suppose the assertion holds upto $m-1$. 
	
	If possible, suppose there exists a non-trivial exponential map $\phi$ on $B$ such that $u \in B^{\phi}$. By \thref{dhm}, with respect to the filtration induced by $(-1,\ldots,-1) \in \mathbb{Z}^m$ on $B$, we get a non-trivial exponential map $\overline{\phi}$ on the associated graded ring $\overline{B}$, which is isomorphic to $B$ itself (cf. \thref{gr}), and $\overline{u} \in \overline{B}^{\overline{\phi}}$, where $\overline{u}$ denotes the image of $u$ in $\overline{B}$. Therefore, we can assume that $B$ is a graded ring and $\phi$ is a homogeneous exponential map on $B$, such that $u \in B^{\phi}$ and the weights of the generators of $B$ are as follows: 
	$$
	wt(t_i)=-1, \text{~for every~} i, 1 \leqslant i \leqslant m,\, wt(u)=r_1+\cdots+r_m,\, wt(v)=0.
	$$
    We first show that $B^{\phi} \nsubseteq k[u,v]$. Suppose, if possible, $B^{\phi} \subseteq k[u,v] (\subseteq B)$. This can happen only when $m=2$ and hence $\td_{k} B^{\phi}=2$.
    Thus $B^{\phi}=k[u,v]$ (cf. \thref{prop}(i)). But then, it follows that $t_1,t_2 \in B^{\phi}$, as $t_1^{r_1} t_2^{r_2}u=P(v) \in B^{\phi}$ (\thref{prop}(i)). This contradicts the fact that $\phi$ is non-trivial. Hence $B^{\phi} \nsubseteq k[u,v]$.
    
    Therefore, there exists $g \in B^{\phi}\setminus k[u,v]$, which is homogeneous with respect to the grading on $B$ and 
	$$
	g= \widetilde{g}(t_1,\ldots,t_m,v)+ \sum_{(\underline{i},j,q) \in \Lambda} \alpha_{\underline{i}\, jq}\, t_{1}^{i_{1}} \cdots t_{m}^{i_{m}} u^j v^q ,
	$$
	where $\widetilde{g} \in k^{[m+1]}$ and $\alpha_{\underline{i}\, jq} \in k$. Now the following two cases can occur. We choose a suitable index $l, 1 \leqslant l \leqslant m$ as follows:
	
	\smallskip
	\noindent
	{\it Case  }1:
	If $\widetilde{g} \notin k[v]$, then it has a monomial summand $g_2$ such that $t_l \mid g_2$, for some $l \in \{1,\ldots,m\}$.
	
	\medskip
	\noindent
	{\it Case }2:
	If $\widetilde{g} \in k[v]$, then there exist at least one nonzero summand of $g$ of the form $\alpha_{\underline{i}\, jq}\, t_{1}^{i_{1}} \cdots t_{m}^{i_{m}} u^j v^q$, and each such summand has weight zero. We fix such a summand. Since its weight is zero, we have $j(r_{1}+\cdots+r_{m})=i_{1}+\cdots+i_{m}$. 
	Also there exists some $s \in \{1,\ldots, m\}$ such that $i_s<r_s$.
	Hence there exists some $l \in \{1,\ldots,m\}, l \neq s$ such that $i_{l}> jr_l$.  
	
	Now if we consider the filtration on $B$, induced by $(0,\ldots,0,1,0,\ldots,0) \in \mathbb{Z}^m$, where the $l$-th entry is 1, then $\phi$ will induce a non-trivial exponential map $\widehat{\phi}$ on the associated graded ring (cf. \thref{gr})
	$$
	\widehat{B} \cong B = \frac{k[T_1,\ldots,T_m,U,V]}{(T_1^{r_1}\cdots T_m^{r_m}U-P(V))}.
	$$
	 For every $b \in B$, let $\widehat{b}$ denote its image in $\widehat{B}$. Note that $\widehat{u}, \widehat{g} \in \widehat{B}^{\widehat{\phi}}$. Further, one can see from {\it Cases} 1 and 2 that $\widehat{t_l} \mid \widehat{g}$. Hence  
	 $\widehat{t_l} \in \widehat{B}^{\widehat{\phi}}$ (cf. \thref{prop}(i), \thref{dhm}). Therefore, by \thref{prop}(iii), $\widehat{\phi}$ will induce a non-trivial exponential map $\widetilde{\phi}$ on 
	$$
	\widetilde{B}= \widehat{B} \otimes_{k[\widehat{t_l}]}k(\widehat{t_l}) \cong \frac{k(T_l)[T_1,\ldots,T_{l-1},T_{l+1},\ldots,T_m,U,V]}{(T_1^{r_1}\cdots T_{l-1}^{r_{l-1}} T_{l+1}^{r_{l+1}} \cdots T_m^{r_m}U-P(V))}.
	$$
	Since $\widehat{u} \in \widehat{B}^{\widehat{\phi}}$, we have $\widetilde{u} \in \widetilde{B}^{\widetilde{\phi}}$, where $\widetilde{u}$ is the image of $\widehat{u}$ in $\widetilde{B}$. But this contradicts the induction hypothesis.  Hence the result follows.
\end{proof}

The next result describes the Makar-Limanov invariant of the ring $B_{\underline{r}}$.

\begin{thm}\thlabel{mld}
  Let $B_{\underline{r}}$ be the ring as in \eqref{b}.	Then the  following hold:
  \begin{enumerate}
  	\item [\rm(a)] If $\underline{r}\in \mathbb{Z}^m_{>1}$, then $\ml(B_{\underline{r}})=k[t_1,\ldots,t_m]$. 
  	
  	\item [\rm (b)] If $\underline{r}\in \mathbb{Z}^m_{\geqslant 1} \setminus \mathbb{Z}^m_{>1}$, then $\ml(B_{\underline{r}}) \subsetneq k[t_1,\ldots,t_m]$, and for $\underline{1}=(1,\ldots,1)$, $\ml(B_{\underline{1}})=k$.
  \end{enumerate}
\end{thm}

\begin{proof}
	(a) We first show that for any non-trivial exponential map $\phi$ on $B_{\underline{r}}$, $B_{\underline{r}}^{\phi} \subseteq k[t_1,\ldots,t_m]$. 
Suppose, if possible, there exists a non-trivial exponential map $\psi$ 
on	$B_{\underline{r}}$ such that $B_{\underline{r}}^{\psi} \nsubseteq k[t_1,\ldots,t_m]$. Therefore, there exists $g \in B_{\underline{r}}^{\psi} \setminus k[t_1,\ldots,t_m]$.
Further, suppose that $g \notin k[t_1,\ldots,t_m,v]$. Then, $g$ can be uniquely expressed as
    $$
	g= g_1(t_1,\ldots,t_m,v)+ \sum_{(\underline{i},j,q) \in \Lambda} \alpha_{\underline{i}\, jq}\, t_{1}^{i_{1}} \cdots t_{m}^{i_{m}} u^j v^q ,
	$$
	where $ g_1 \in k^{[m+1]}$ and $\alpha_{\underline{i}\, jq} \in k^{*}$.

	Let us choose a summand $\alpha_{\underline{i}\, jq}t_1^{i_1}\cdots t_m^{i_m}u^jv^q$ of $g$, where $i_s<r_s$ for some $s, 1\leqslant s \leqslant m$. We now consider the proper $\mathbb{Z}$-filtration on $B_{\underline{r}}$, induced by $(0,\ldots,0,-1,0,\ldots,0) \in \mathbb{Z}^m$, where the $s$-th entry is $-1$. Then $\psi$ induces a non-trivial exponential map $\overline{\psi}$ on the associated graded ring $\overline{B}_{\underline{r}}$. By \thref{gr},
	$$
	\overline{B}_{\underline{r}} \cong \frac{k[T_1,\ldots,T_m,U,V]}{(T_1^{r_1}\cdots T_m^{r_m}U-F(T_1,\ldots,T_{s-1},0,T_{s+1},\ldots,T_m,V))}.
	$$
	 For every $b\in B_{\underline{r}}$, let $\overline{b}$ denote the image of $b$ in $\overline{B}_{\underline{r}}$. With respect to the chosen filtration it is clear that $\overline{u} \mid \overline{g}$ and since $\overline{g} \in \overline{B}^{\overline{\psi}}_{\underline{r}}$ (\thref{dhm}), $\overline{u} \in \overline{B}^{\overline{\psi}}_{\underline{r}}$.
	 
	  Further with respect to the $\mathbb{Z}$-filtration induced by $(-1,\ldots,-1) \in \mathbb{Z}^m_{>1}$ on $\overline{B_{\underline{r}}}$, $\overline{\psi}$ induces a non-trivial exponential map $\psi^{\prime}$ on the associated graded ring $B_{\underline{r}}^{\prime}$. By \thref{gr},
	 $$
	 B_{\underline{r}}^{\prime} \cong \frac{k[T_1,\ldots,T_m,U,V]}{(T_1^{r_1}\cdots T_m^{r_m}U-P(V))}.
	 $$ 
	  Let $u^{\prime}$ denote the image of $U$ in $B_{\underline{r}}^{\prime}$.  Since $\overline{u} \in \overline{B}^{\overline{\psi}}_{\underline{r}}$, $u^{\prime} \in (B_{\underline{r}}^{\prime})^{\psi^{\prime}}$.
	 But this contradicts \thref{m2}. Therefore, we have $g \in k[t_1,\ldots,t_m,v]=R$.
	
	\medskip
	Now note that $B_{\underline{r}} \hookrightarrow C:=k[t_1^{\pm 1},\ldots,t_m^{\pm 1},v]$. Set $D_n:= k[t_1^{\pm 1}, \ldots, t_m^{\pm 1}]v^n$ for all $n \ge 0$.
	The ring $k[t_1^{\pm 1},\ldots,t_m^{\pm 1},v]$ can be given the following $\mathbb{Z}$-graded structure:
	$$
	k[t_{1}^{\pm 1},\ldots,t_{m}^{\pm 1},v] = \bigoplus_{n \geqslant 0} D_{n}=\bigoplus_{n \geqslant 0} k[t_1^{\pm 1}, \ldots, t_m^{\pm 1}]v^n.
	$$
	This induces a proper $\mathbb{Z}$-filtration $\{(B_{\underline{r}})_n\}_{n \in \mathbb{Z}}$ on $B_{\underline{r}}$ such that $(B_{\underline{r}})_{n} = \left(\bigoplus_{i \leqslant n} D_{n}\right) \cap B_{\underline{r}}$. 
	 Set 
	\begin{equation}\label{Lambda}
		\Lambda_1:=\biggl\{ (\underline{i},j):=(i_1,\ldots,i_m,j) \in \mathbb{Z}^m_{\geqslant0} \times \mathbb{Z}_{>0} \mid i_s <r_s \text{~for some~} s, 1\leqslant s \leqslant m \biggr\}.
	\end{equation}
	 Using the relation $t_1^{r_1}\cdots t_m^{r_m}u=F(t_1,\ldots,t_m,v)$, one can see that every element $b \in B_{\underline{r}}$ can be uniquely expressed as
	\begin{equation}\label{b2}
		b= \sum_{n \geqslant 0}^{}b_n (t_1,\ldots,t_m) v^n+ \sum_{(\underline{i},j)\in \Lambda_1} b_{\underline{i}\,j}(v)\, t_{1}^{i_{1}} \cdots t_{m}^{i_{m}} u^j  ,
	\end{equation}
	such that $b_{\underline{i}j}(v) \in k[v] \setminus \{0\}$.
	
	Now since the filtration $\{(B_{\underline{r}})_n\}_{n \in \mathbb{Z}}$ on $B_{\underline{r}}$ is induced from the graded structure of the ring $C$, from the expression \eqref{b2}, it follows that the filtration $\{(B_{\underline{r}})_n\}_{n \in \mathbb{Z}}$ is admissible with respect to the generating set $\Gamma=\{t_1,\ldots,t_m,u,v\}$ of $B_{\underline{r}}$ and the associated graded ring 
	$$
	E:=\bigoplus_{n \in \mathbb{Z}} \frac{B_n}{B_{n-1}}
	$$ 
	is generated by the image of $\Gamma$ in $E$. 
	For any $b \in B_{\underline{r}}$, let $\widetilde{b}$ denote its image in $E$. Note that $\widetilde{t_1}^{r_1} \cdots \widetilde{t_m}^{r_m} \widetilde{u}= \widetilde{v}^d$ in $E$, where $d=\deg_V F=\deg_V P(V) (>1)$. 

	As $E$ can be identified with a subring of the graded domain $gr(C) \cong k[t_1^{\pm 1},\ldots,t_m^{\pm 1}, v]$, we have $\widetilde{t_1},\ldots,\widetilde{t_m},\widetilde{v}$ are algebraically independent in $E$ and hence $\dim E=m+1$.
	Since $\frac{k[T_1,\ldots,T_m,U,V]}{(T_1^{r_1}\cdots T_m^{r_m}U- V^d)} $ is an integral domain of dimension $m+1$, we have the following isomorphism
	\begin{equation}\label{e}
		E \cong \frac{k[T_1,\ldots,T_m,U,V]}{(T_1^{r_1}\cdots T_m^{r_m}U- V^d)}.
	\end{equation}
	Now by \thref{dhm}, we have $\psi$ induces a non-trivial exponential map $\widetilde{\psi}$ on $E$ such that $\widetilde{g} \in E^{\widetilde{\psi}}$. Now from the grading on $E$ it is clear that $\widetilde{v} \mid \widetilde{g}$ and hence $\widetilde{v} \in E^{\widetilde{\psi}}$. But then from \eqref{e} it follows that $\widetilde{t_1},\ldots,\widetilde{t_m},\widetilde{u}\in E^{\widetilde{\psi}}$ (cf. \thref{prop}(i)) and hence $\widetilde{\psi}$ is a trivial exponential map, which is a contradiction.
	
	Therefore we obtain that for every non-trivial exponential map $\phi$ on $B_{\underline{r}}$, $B_{\underline{r}}^{\phi} \subseteq k[t_1,\ldots,t_m]$. Since $B_{\underline{r}}^{\phi}$ is algebraically closed in $B_{\underline{r}}$ and $\td_{k} B^{\phi}_{\underline{r}} =m$ (cf. \thref{prop}(ii)), we have $B_{\underline{r}}^{\phi}=k[t_1,\ldots,t_m]$. Therefore $\ml(B_{\underline{r}})=k[t_1,\ldots,t_m]$. 
	
	\medskip
	\noindent
	(b) Let $\underline{r}=(r_1,\ldots,r_m) \in \mathbb{Z}^m_{\geqslant 1}\setminus \mathbb{Z}^m_{>1} $. Suppose $r_j =1$ for some $j, 1 \leqslant j \leqslant m$. Note that 
	$$
		B_{\underline{r}} = \frac{k[T_1,\ldots,T_m,U,V]}{(T_1^{r_1}\cdots T_{j-1}^{r_{j-1}} T_{j+1}^{r_{j+1}}\cdots T_m^{r_m} T_j U-F(T_1,\ldots,T_m,V))}.
	$$
	Let $E_j:=B_{\underline{r}}[t_1^{-1},\ldots,t_{j-1}^{-1},t_{j+1}^{-1},\ldots, t_{m}^{-1}]$ and $C_j:=k[t_1^{\pm 1},\ldots,t_{j-1}^{\pm 1},t_{j+1}^{\pm 1},\ldots,t_m^{\pm 1}]$. Suppose that
	$$
	F=T_j F_j(T_1,\ldots,T_m,V)+ F(T_1,\ldots,T_{j-1},0,T_{j+1},\ldots,T_m,V).
	$$
	Now for 
	\begin{equation}\label{uj}
	u_j:=u-\frac{F_j(t_1,\ldots,t_m,v)}{t_1^{r_1}\cdots t_{j-1}^{r_{j-1}}t_{j+1}^{r_{j+1}}\cdots t_m^{r_m}} \in E_j,
    \end{equation}
	we have 
	$$
		t_j u_j -\frac{F(t_1,\ldots,t_{j-1},0,t_{j+1},\ldots,t_m,v)}{t_1^{r_1}\cdots t_{j-1}^{r_{j-1}}t_{j+1}^{r_{j+1}}\cdots t_m^{r_m}}=0
	$$
	and $E_j=C_j[t_j, u_j, v]$. We consider the exponential map $\phi_{j}: E_j \rightarrow E_j[W]$ such that
	$$
	\phi_j|_{C_j}= id_{C_j},~~\phi_j(u_j)=u_j, ~~ \phi_j(v)=v+u_j W \text{~and~}
	$$ 
	$$
	\phi_j(t_j)=\frac{F(t_1,\ldots,t_{j-1},0,t_{j+1},\ldots,t_m,v+u_jW)}{t_1^{r_1}\cdots t_{j-1}^{r_{j-1}}t_{j+1}^{r_{j+1}}\cdots t_m^{r_m}u_j}=t_j+\sum_{i=1}^{r}\alpha_i W^i,
	$$
	where $\alpha_i \in E_j$, for each $i, 1 \leqslant i \leqslant r$. Now since $\phi_{j}(u_j)=u_j$, using \eqref{uj}, it follows that 
	$$
	\phi_j(u)=u+\frac{F_j(t_1,\ldots,t_j+\sum_{i=1}^{r}\alpha_i 
	W^i,\ldots,t_m,v+u_jW)-F_j(t_1,\ldots,t_m,v)}{t_1^{r_1}\cdots 
	t_{j-1}^{r_{j-1}}t_{j+1}^{r_{j+1}}\cdots t_m^{r_m}}= u+\sum_{l=1}^{s}\beta_l W^l,
	$$
	where $\beta_l \in E_j$ for every $l, 1 \leqslant l \leqslant s$.
	 
	 Let $p_j:= t_1\cdots t_{j-1}t_{j+1}\cdots t_m$ and $n$ be the smallest positive integer such that $p_j^n u_j, p_j^n \alpha_i , p_j^n \beta_l \in B_{\underline{r}}$ for every $i, 1 \leqslant i \leqslant r$ and every $l, 1 \leqslant l \leqslant s$. Since $\phi_{j}(p_j)=p_j$, $\phi_j$ induces an exponential map $\widetilde{\phi_j}: B_{\underline{r}} \rightarrow B_{\underline{r}}[W]$ such that 
	 \begin{align*}
	 	&\widetilde{\phi_j}(t_i)=t_i, \text{~for every~} i, 1\leqslant i \leqslant m \text{~and~} i \neq j,\\
	 	&\widetilde{\phi_j}(t_j)=t_j+ \sum_{i=1}^{r} \alpha_i (p_j^nW)^i,\\
	 	& \widetilde{\phi_j}(v)=v+u_jp_j^n W \text{~and~}\\ 
	 	& \widetilde{\phi_j}(u)=u+ \sum_{l=1}^{s} \beta_l (p_j^nW)^l.
	 \end{align*}
	Since $t_1,\ldots,t_{j-1},t_{j+1},\ldots,t_m,u_j \in E_j^{\phi_{j}}$, it is clear that $\widetilde{u_j} \in E_j^{\phi_j}$, where
   	$$
   	\widetilde{u_j}:=t_1^{r_1}\cdots 
   	t_{j-1}^{r_{j-1}}t_{j+1}^{r_{j+1}}\cdots t_m^{r_m}u_j= t_1^{r_1}\cdots 
   	t_{j-1}^{r_{j-1}}t_{j+1}^{r_{j+1}}\cdots t_m^{r_m}u -F_j(t_1,\ldots,t_m,v).
   	$$ 
	Therefore, from the definition of $\widetilde{\phi_j}$, it follows that $\widetilde{u_j} \in B_{\underline{r}}^{\widetilde{\phi_j}}$, and hence
	$$
	D_j:=k[t_1,\ldots,t_{j-1},\widetilde{u_j},t_{j+1},\ldots,t_m] \subseteq B_{\underline{r}}^{\widetilde{\phi_j}}.
	$$
	Further, since $\td_{k}D_j=\td_k B_{\underline{r}}^{\widetilde{\phi_j}}$ and $D_j$ is algebraically closed in $B_{\underline{r}}^{\widetilde{\phi_j}}$, we have $B_{\underline{r}}^{\widetilde{\phi_j}}=D_j$.
	 
\medskip
	Again consider the following map $\phi: B_{\underline{r}} \rightarrow B_{\underline{r}}[W]$ such that 
	$$
	\phi(t_i)=t_i \text{~for every~} i,1 \leqslant i \leqslant m,~~ \phi(v)=v+ t_1^{r_1}\cdots t_m^{r_m} W
	$$
	and 
	$$
	\phi(u)= \frac{F(t_1,\ldots,t_m,v+ t_1^{r_1}\cdots t_m^{r_m} W)}{t_1^{r_1}\cdots t_m^{r_m} }= u+ W \alpha(t_1, \ldots, t_m, v, W),
	$$
	where $\alpha \in k^{[m+2]}$. It is easy to see that $\phi \in \text{EXP}(B_{\underline{r}})$ and $k[t_1,\ldots,t_m] \subseteq B_{\underline{r}}^{\phi}$. As $k[t_1,\ldots,t_m]$ is algebraically closed in $B_{\underline{r}}$, we have $B_{\underline{r}}^{\phi}=k[t_1,\ldots,t_m]$.   
	
	 Therefore, we obtain that $ \ml(B_{\underline{r}}) \subseteq B_{\underline{r}}^{\phi} \bigcap_{\{j \,\mid\, r_j=1\}} B_{\underline{r}}^{\widetilde{\phi_{j}}} \subsetneq k[t_1,\ldots,t_m]$. In particular, for $\underline{r}=\underline{1}$ we have $k \subseteq \ml(B_{\underline{1}}) \subseteq B_{\underline{1}}^{\phi} \bigcap_{1 \leqslant j \leqslant m} B_{\underline{1}}^{\widetilde{\phi_{j}}}=k$. Hence the result follows.
\end{proof}

\begin{rem}\thlabel{a}
	\em{From \thref{mld}, it is clear that when $\underline{r} \in \mathbb{Z}^m_{>1}$ and $\underline{s} \in \mathbb{Z}^m_{\geqslant 1} \setminus \mathbb{Z}^m_{>1}$, then $B_{\underline{r}} \ncong B_{\underline{s}}$, as $\ml(B_{\underline{r}}) \neq \ml(B_{\underline{s}})$.}
\end{rem}

The next theorem classifies the generalised Danielewski varieties $B_{\underline{r}}$ upto isomorphism when $\underline{r} \in \mathbb{Z}^m_{>1}$.

\begin{thm}\thlabel{isod}
	Let $(r_1,\ldots,r_m),(s_1,\ldots,s_m) \in \mathbb{Z}^m_{>1}$ and $F,G \in k[T_1,\ldots,T_m,V]$ be monic polynomials in $V$ each of degree more than $1$, such that $P(V)=F(0,\ldots,0,V)$ and $Q(V)=G(0,\ldots,0,V)$.
	Suppose 
	$$
	B:=B(r_1,\ldots,r_m,F)=\frac{k[T_1,\ldots,T_m,U,V]}{(T_1^{r_1}\cdots T_m^{r_m}U-F(T_1,\ldots,T_m,V))}
	$$ 
	and 
	$$
	B^{\prime}:=B(s_1,\ldots,s_m,G)=\frac{k[T_1,\ldots,T_m,U,V]}{(T_1^{s_1}\cdots T_m^{s_m}U-G(T_1,\ldots,T_m,V))}
	$$ 
   If $B(r_1,\ldots,r_m,F) \cong B(s_1,\ldots,s_m,G)$ then 
	
	\begin{enumerate}
		\item [\rm(i)] $(r_1,\ldots,r_m)=(s_1,\ldots,s_m)$ upto a permutation of $\{1,\ldots,m\}$.
		
		\item[\rm(ii)] There exists $\alpha \in Aut_k(k[V])$ such that $\alpha(Q) = \lambda P$, for some $\lambda \in k^{*}$. Thus $\deg_V P= \deg_V Q$.
	\end{enumerate}

	Furthermore, if $F,G \in k[V]$, then the conditions (i) and (ii) are sufficient for  $B(r_1,\ldots,r_m,F)$ to be isomorphic to $B(s_1,\ldots,s_m,G)$.   
\end{thm}

\begin{proof}

	  (i) Let $t_1,\ldots,t_m,u,v$ and $t_1^{\prime},\ldots,t_m^{\prime},u^{\prime},v^{\prime}$ denote the images of $T_1,\ldots,T_m,U,V$ in $B$ and $B^{\prime}$ respectively.
	   Let $\rho: B \rightarrow B^{\prime}$ be an isomorphism. Identifying $\rho(B)$ by $B$, we assume that $B^{\prime}=B$. 
	   By \thref{mld}(a), we have 
	$$
	\ml(B)=k[t_1,\ldots,t_m]=k[t_1^{\prime},\ldots,t_m^{\prime}].
	$$ 
	Therefore,
	$$
	B\otimes_{k[t_1,\ldots,t_m]}k(t_1,\ldots,t_m)=B\otimes_{k[t_1^{\prime},\ldots,t_m^{\prime}]} k(t_1^{\prime},\ldots,t_m^{\prime}),
	$$ 
	and hence 
	\begin{equation}\label{z}
		k(t_1,\ldots,t_m)[v]=k(t_1^{\prime},\ldots,t_m^{\prime})[v^{\prime}]=k(t_1,\ldots,t_m)[v^{\prime}].
	\end{equation}
	Note that $R:=k[t_1,\ldots,t_m,v] \hookrightarrow B \hookrightarrow k[t_1^{\pm 1},\ldots,t_m^{\pm 1},v]$. We now show that $v^{\prime} \in R$. Suppose $v^{\prime} \in B \setminus R$. Then 
	\begin{equation}\label{v}
		v^{\prime}= g(t_1,\ldots,t_m,v)+ \sum_{(\underline{i},j) \in \Lambda_1} b_{\underline{i}\,j}(v)\, t_{1}^{i_{1}} \cdots t_{m}^{i_{m}} u^j  ,
	\end{equation}
	where $\Lambda_1$ is as in \eqref{Lambda}, $g \in R$ and $b_{\underline{i}\,j}(v) \in k[v] \setminus\{0\}$. 
	Since $u=\frac{F(t_1,\ldots,t_m,v)}{t_1^{r_1}\cdots t_m^{r_m}}$ and $\deg_vP(v)>1$, from \eqref{v}, it is clear that $\deg_v v^{\prime} >1$, when considered as an element in $k(t_1,\ldots,t_m)[v]$. But this contradicts \eqref{z}. Therefore, $v^{\prime} \in R$. Now using the symmetry in \eqref{z}, we obtain that 
	\begin{equation}\label{r}
		R=k[t_1,\ldots,t_m,v]=k[t_1^{\prime},\ldots,t_m^{\prime},v^{\prime}].
	\end{equation}
    Also from \eqref{z} and \eqref{r}, it is clear that 
    \begin{equation}\label{v2}
    v^{\prime}= \gamma v+ f(t_1,\ldots,t_m),
    \end{equation}	 
    for some $\gamma \in k^{*}$ and $f \in k^{[m]}$.
    Now
	$$
	u^{\prime} =\frac{G(t_1^{\prime},\ldots,t_m^{\prime},v^{\prime})}{(t_1^{\prime})^{s_1}\cdots(t_m^{\prime})^{s_m} } \in B \setminus R.
	$$ 
	 Since $B \hookrightarrow k[t_1^{\pm 1},\ldots,t_m^{\pm 1},v]$, there exists $n>0$ such that 
	$$
	(t_1\cdots t_m)^n u^{\prime} = \frac{(t_1\cdots t_m)^n G(t_1^{\prime},\ldots,t_m^{\prime},v^{\prime})}{ (t_1^{\prime})^{s_1}\cdots(t_m^{\prime})^{s_m}} \in R.
	$$
	Since for every $i \in \{1,\ldots,m\}$, $t_i^{\prime}$ is irreducible in $R$ and $t_{i}^{\prime} \nmid G$, we have $t_i^{\prime} \mid (t_1\ldots t_m)$. As $t_1,\ldots,t_m$ are also irreducibles in $R$, we have
	\begin{equation}\label{t}
		t_i^{\prime}=\lambda_jt_j,
	\end{equation}
	for some $j\in \{1,\ldots,m\}$ and $\lambda_{j} \in k^{*}$. We now show that $s_i=r_j$. Suppose $s_i>r_j$. Consider the ideal 
	$$
	\mathfrak{a}_i:= (t_i^{\prime})^{s_i}B \cap R = \left( (t_i^{\prime})^{s_i}, G(t_1^{\prime},\ldots,t_m^{\prime},v^{\prime}) \right).
	$$
	 Again 
	 $$
	 \mathfrak{a}_i= t_j^{s_i}B \cap R = \left( t_j^{s_i}, t_j^{s_i-r_j}F(t_1,\ldots,t_m,v) \right),
	 $$
	 which implies that $G(t_1^{\prime},\ldots,t_m^{\prime},v^{\prime}) \in t_jR$. But this is a contradiction. Therefore, $s_i \leqslant r_j$. 
	
	Again by similar arguments as above, we get that $s_i \geqslant r_j$. 
 Therefore, we have $s_i=r_j$. Hence it follows that $(r_1,\ldots,r_m)=(s_1,\ldots,s_m)$ upto a permutation. 
	
	\medskip
	\noindent
	(ii) 	Since  by \eqref{t} $\left(t_1 \cdots t_m\right)B \,\cap\, R = \left(t_1^{\prime} \cdots t_m^{\prime}\right)B \,\cap\, R$, we have 
	$$
	\left(t_1\cdots t_m, F\right)R=\left(t_1^{\prime}\cdots t_m^{\prime}, G\right)R.
	$$
	Therefore, it follows that 
	\begin{equation}\label{q}
		Q(v^{\prime})= \lambda P(v)+ Q_1(t_1,\ldots,t_m,v),
	\end{equation}
	for some $\lambda \in k^{*}$ and $Q_1 \in (t_1,\ldots,t_m)R$. Therefore using \eqref{v2}, one can see that $\alpha: k[V] \rightarrow k[V]$ defined by $\alpha(V)=\gamma V+f(0,\ldots,0)$ is an automorphism of $k[V]$, and from \eqref{q} it follows that $\alpha(Q)=\lambda P$. 
	
	The converse is obvious.
\end{proof}
We now record an elementary lemma.

\begin{lem}\thlabel{ex1}
	Let $E,D$ be integral domains such that $E \subseteq D$. Suppose there exists $a (\neq 0) \in E$ such that $E[a^{-1}]=D[a^{-1}]$ and $aD \cap E=aE$. Then $E=D$.
\end{lem}
The next theorem exhibits a certain sub-family of generalised Danielewski varieties which are stably isomorphic.
\begin{thm}\thlabel{stiso}
	Let $(r_1,\ldots,r_m),(s_1,\ldots,s_m) \in \mathbb{Z}^m_{\geqslant 1}$. 
	 If $\left(F,F_V\right)=k[T_1,\ldots,T_m,V]$, then 
	$$
	B(r_1,\ldots,r_m,F)^{[1]} \cong B(s_1,\ldots,s_m,F)^{[1]}.
	$$
\end{thm}
\begin{proof}
	Let 
	$
	B(r_1,\ldots,r_m,F) 
	$
	be such that $r_j>1$ for some $j \in \{1,\ldots,m\}$. Without loss of generality we assume $r_1>1$. We now show that 
	$$
	B(r_1-1,r_2,\ldots,r_m,F)^{[1]} \cong B(r_1,\ldots,r_m,F)^{[1]},
	$$
	and therefore for any pair $(r_1,\ldots,r_m), (s_1,\ldots,s_m) \in \mathbb{Z}^m_{\geqslant 1}$, we get the result inductively.
	
	Let $E=B(r_1,\ldots,r_m,F)[w]=B(r_1,\ldots,r_m,F)^{[1]}$. Consider the exponential map $\phi: E \rightarrow E[T]=E^{[1]}$ as follows:
	\begin{align*}
		\phi(t_i)&=t_i, \text{~for all~} i,1 \leqslant i \leqslant m,\\
		 \phi(v) &= v +t_1^{r_1}\cdots t_m^{r_m} T,\\
		 \phi(u)&= \frac{F(t_1,\ldots,t_m,v+t_1^{r_1}\cdots t_m^{r_m}T)}{t_1^{r_1}\cdots t_m^{r_m}}= u+T \alpha (t_1,\ldots ,t_m,v,T),\\
		 \phi(w)&=w-t_1T,
	\end{align*}
	where $\alpha \in k^{[m+2]}$. Now for 
	\begin{equation}\label{v1}
	v_1=v+t_1^{r_1-1} t_2^{r_2}\cdots t_m^{r_m}w,
	\end{equation}
	we have $v_1 \in E^{\phi}$. Again,
	\begin{align*}
		F(t_1,\ldots,t_m,v_1)&=F(t_1,\ldots,t_m,v+t_1^{r_1-1}\cdots t_m^{r_m}w)\\
		~~~~~~~~~~~~~~~~~~~~~&=F(t_1,\ldots,t_m,v)+ t_1^{r_1-1}\cdots t_m^{r_m}(wF_V(t_1,\ldots,t_m,v)+bt_1),
	\end{align*}
for some $b \in k[t_1, \dots, t_m, v, w]$. 
Therefore, as $F(t_1,\ldots,t_m,v)= t_1^{r_1}\cdots t_m^{r_m} u$, we have
\begin{equation}\label{P}
 F(t_1,\ldots,t_m,v_1)= t_1^{r_1-1}\cdots t_m^{r_m} u_1
\end{equation}
	where 
	\begin{equation}\label{u}
		u_1=t_1u+ w F_V(t_1,\ldots,t_m,v)+bt_1.
	\end{equation}
	 Since $t_1,\ldots,t_m,v_1 \in E^{\phi}$, by (\ref{P}), $u_1\in E^{\phi}$ (cf. \thref{prop}(i)). Now since $\left(F,F_V\right)=k[T_1,\ldots,T_m,V]$, there exist $g_1,g_2 \in k^{[m+1]}$ such that 
	\begin{equation}\label{p}
		F(t_1,\ldots,t_m,v)g_1(t_1,\ldots,t_m,v)+F_V(t_1,\ldots,t_m,v)g_2(t_1,\ldots,t_m,v)=1.
	\end{equation}
	Note that $g_2(t_1,\ldots,t_m,v_1)-g_2(t_1,\ldots,t_m,v) \in t_1E$. Therefore,
	\begin{align*}
		w-u_1g_2(t_1,\ldots,t_m,v_1)&=w-g_2(t_1,\ldots,t_m,v_1)(t_1u+ w F_V(t_1,\ldots,t_m,v)+bt_1)\\\nonumber
		&= w(1-F_V(t_1,\ldots,t_m,v)g_2(t_1,\ldots,t_m,v_1)) +t_1 \theta\\\nonumber
		&= w(1-F_V(t_1,\ldots,t_m,v)g_2(t_1,\ldots,t_m,v)) +t_1 \delta\\\nonumber
		&= wF(t_1,\ldots,t_m,v)g_1(t_1,\ldots,t_m,v) +t_1 \delta\\\nonumber
        &= w  t_1^{r_1}\cdots t_m^{r_m} u  g_1(t_1,\ldots,t_m,v) +t_1 \delta\\\nonumber
		&=t_1 \widetilde{w},
	\end{align*}
	where $\theta, \delta\in E$ and 
	$ \widetilde{w}= wu t_1^{r_1-1}\cdots t_m^{r_m}  g_1(t_1,\ldots,t_m,v) +\delta \in E$. Therefore, we have 
\begin{equation}\label{w1}
\widetilde{w}=\frac{w-u_1g_2(t_1, \dots, t_m,v_1)}{t_1} \in E.
\end{equation}
  Note that $\phi(\widetilde{w})=\widetilde{w} -T$ and hence, by \thref{prop}(ii), $E=E^{\phi}[\widetilde{w}]=(E^{\phi})^{[1]}$. 
	Let $C:=k[t_1,\ldots,t_m,v_1,u_1]$. Clearly $C \subseteq E^{\phi}$.
	By \eqref{P}, $\td_k\, C =m+1$ and hence $\dim \,C=m+1$.
We show that	 

	\medskip
	
	\noindent
	(a) $C \cong B(r_1-1,r_2\ldots,r_m,F) \cong \frac{k[T_1,\ldots,T_m,U,V]}{(T_1^{r_1-1}T_2^{r_2}\cdots T_m^{r_m}U-F(T_1,\ldots,T_m,V))},$
	
	\medskip
	\noindent
	(b) $C=E^{\phi}$.
	
	\medskip
\noindent
(a) Consider the surjection $ \psi : k[T_1,\ldots,T_m,V,U] \rightarrow C$ such that 
	$$
	\psi(T_i)=t_i \text{~for all~}i,1 \leqslant i \leqslant m, ~ \psi(V)=v_1,~ \psi(U)=u_1.
	$$
	 From \eqref{P} it is clear that $(T_1^{r_1-1}T_2^{r_2} \cdots T_m^{r_m}U-F(T_1,\ldots,T_m,V)) \subseteq \ker \psi$. Therefore, $\psi$ induces a surjection
	$$
	\overline{\psi} : B(r_1-1,r_2,\ldots,r_m,F) =\frac{k[T_1,\ldots,T_m,U,V]}{(T_1^{r_1-1}T_2^{r_2}\cdots T_m^{r_m}U-F(T_1,\ldots,T_m,V))} \longrightarrow C.
	$$
	Since $B(r_1-1,r_2,\ldots,r_m,F)$ is an integral domain and $\dim\, C=m+1$, we have $\overline{\psi}$ is an isomorphism.
	
	\medskip
	\noindent
(b) We note that 
\begin{align*}
E[t_1^{-1}] &= k[t_1^{\pm 1}, \dots, t_m, u, v, w] \\
&=k[t_1^{\pm 1}, \dots, t_m, u_1, v_1, w] \text{~~~by~~~} \eqref{v1} \text{~and~} \eqref{u}\\
&=k[t_1^{\pm 1}, \dots, t_m, u_1, v_1, \widetilde{w}] \text{~~~by~~~} \eqref{w1}\\
&= C[t_1^{-1}][\widetilde{w}]\\
&= E^{\phi}[t_1^{-1}][\widetilde{w}]
\end{align*}
Since $C \subseteq E^{\phi}$, it follows that $C[t_1^{-1}]= E^{\phi}[t_1^{-1}]$. 
Therefore, to show that $C= E^{\phi}$, by \thref{ex1}, it is enough to show that $t_1E^{\phi}\cap C = t_1C$. Since $t_1E \cap E^{\phi}= t_1E^{\phi}$, 
it is therefore enough to show that $t_1E \cap C= t_1 C$, i.e., the kernel of the map 
$\pi: C \to E/t_1E$ is $t_1C$. For every $b \in E$, 
let $\overline{b}$ denote its image in $E/t_1E$, and for every $c \in C$, let $\widehat{c}$ denote its image in $C/t_1C$.
We note that by \eqref{P} and (a), 
\begin{equation}\label{C/xC}
	C/t_1C =k[\widehat{t_2},\ldots,\widehat{t_m},\widehat{v_1}, \widehat{u_1}] \cong \left( \frac{k[T_2,\ldots,T_m,V]}{(F(0,T_2,\ldots,T_m,V))} \right)^{[1]}.
\end{equation}
Also,
\begin{equation}\label{B/xB}
E/t_1E \cong \frac{k[T_2,\ldots,T_m,U,V,W]}{(F(0, T_2,\ldots,T_m,V))}=\left(\frac{k[T_2,\ldots,T_m,V]}{(F(0, T_2,\ldots,T_m,V))}\right)[U, W]=k[\overline{t_2}, \dots, \overline{t_m},\overline{v}, \overline{u}, \overline{w}].
	\end{equation} 
Now by (\ref{v1}), we have $\pi({v_1})=\overline{v}$ as $r_1>1$ and 
by (\ref{u}), $\pi({u_1})= \overline{wF_v(t_1, \dots, t_m, v)}$.  By (\ref{p}), 
$\overline{F_v(t_1, \dots, t_m, v)}$  is a unit in $E/t_1E$ and hence 
$$
\pi(C)= k[\overline{t_2}, \dots, \overline{t_m},\overline{v}, \overline{w}] \cong \left( \frac{k[T_2,\ldots,T_m,V]}{(F(0,T_2,\ldots,T_m,V))} \right)^{[1]}.
$$ 	
Therefore, from \eqref{C/xC} it follows that $\pi$ induces an isomorphism between $C/t_1C$ and $\pi(C)$. Hence kernel of $\pi$ is equal to $t_1C$.

Thus, from (a) and (b) we have
$$
B(r_1,\ldots,r_m,F)[w]=E=(E^{\phi})^{[1]} = B(r_1-1,r_2,\ldots,r_m,F)^{[1]}.
$$
\end{proof}

As a consequence we have an infinite family of examples in arbitrary characteristic which are stably isomorphic but not isomorphic (c.f Question 1).
\begin{cor}
		Let $k$ be any field. For each $n \geqslant 2$, there exists an infinite family of pairwise non-isomophic rings of dimension $n$, which are counterexamples to the Cancellation Problem. 
\end{cor}
\begin{proof}
	Consider the family of rings
	\begin{align*}
			\Omega:=\bigl\{&B(\underline{r},F) \mid \underline{r}:=(r_1,\ldots,r_m) \in \mathbb{Z}^m_{ \geqslant 1},  F \in k[T_1,\ldots,T_m,V] \text{~is monic in~}  V \\
			&\text{~and~} (F,F_V)=k[T_1,\ldots,T_m,V]\bigr\}.
	\end{align*}
For any pair $\underline{r}=(r_1,\ldots,r_m), \underline{s}=(s_1,\ldots,s_m) \in \mathbb{Z}^m_{ \geqslant 1}$, by \thref{stiso},  $B(\underline{r},F)^{[1]} \cong B(\underline{s},F)^{[1]}$. Further,   
by \thref{mld}(b) and \thref{isod}, we get an infinite sub-family of $\Omega$ which contains pairwise non-isomorphic rings. Hence taking $n=m+1$ we get the result.
\end{proof}

\section{ Generalised Asanuma Varieties}

Throughout this section, $A$ will denote the following affine domain:

\begin{equation}\label{AG}
	A= \frac{k\left[ X_{1}, \ldots , X_{m}, Y,Z,T\right]}{(X_{1}^{r_{1}} \cdots X_{m}^{r_{m}}Y -H(X_{1}, \ldots , X_{m}, Z,T))},\,\ r_i > 1 \text{~for all~} i, 1 \leqslant i \leqslant m,
\end{equation}
where  $h(Z,T):=H(0, \ldots, 0, Z,T) \neq ~0$. Let $G:= X_{1}^{r_{1}} \cdots X_{m}^{r_{m}}Y -H$.
As defined in the introduction, rings of type \eqref{AG} are the coordinate rings of the generalised Asanuma varieties. 
Let $x_{1},\ldots, x_{m},y,z,t$ denote the images of $X_{1}, \ldots, X_{m},Y,Z,T$ in $A$, respectively. 
 In this section we will study the {\it Derksen invariant} and the {\it Makar-Limanov invariant} of the affine domain $A$, and establish an equivalent criterion for $A$ to be $k^{[m+2]}$, subject to certain conditions on $H(X_{1},\ldots,X_{m},Z,T)$. We first recall a few results from \cite{asa} which will be used in this section. 

The following result describes $\ml(A)$ when $\dk(A)$ is exactly equal to the subring $k[x_1,\ldots,x_m,z,t]$ of $A$ (\cite[Proposition 3.4]{asa}). 

\begin{prop}\thlabel{ml}
	Let $A$ be the affine domain as in \eqref{AG}. Then the following hold:
	
	\begin{itemize}
		\item [\rm (a)] Suppose, for every $i \in \{1,\ldots,m\}$, $x_i \notin A^{*}$, and $H \notin k[X_1,\ldots,X_m]$. Then $\ml(A) \subseteq k[x_{1},\ldots,x_{m}]$.
		
		\item[\rm (b)]  If $\dk(A)=k[x_{1}, \ldots , x_{m},z,t]$, then $\ml(A)=k[x_{1},\ldots,x_{m}]$.
	\end{itemize}
\end{prop}

	\begin{cor}\thlabel{r3}
	Let $k$ be an infinite field and $A$ be as in \eqref{AG}. Suppose that there is no system of coordinates $\{Z_1,T_1\}$ of $k[Z,T]$ such that $h(Z,T)=a_0(Z_1)+a_1(Z_1)T_1$. Then $\dk(A)=k[x_1,\ldots,x_m,z,t]$ and $\ml(A)=k[x_1,\ldots,x_m]$. 
\end{cor} 
\begin{proof}
	The result follows from Propositions \ref{dk} and \ref{ml}(b).
\end{proof}

Next we state below a reformulation of the result \cite[Lemma 3.13]{asa}; the earlier proof goes through here.

\begin{lem}\thlabel{fdk1}
	Suppose there exists an exponential map $\phi$ on $A$ such that $A^{\phi} \nsubseteq k[x_{1}, \ldots , x_{m},z,t]$. Then there exists an integral domain 
	$$
	\widehat{A} \cong \frac{k[X_{1}, \ldots , X_{m}, Y,Z,T]}{(X_{1}^{r_{1}} \cdots X_{m}^{r_{m}}Y-h( Z,T))}
	$$ 
	and a non-trivial exponential map $\widehat{\phi}$ 
	on $\widehat{A}$, induced by $\phi$, such that 
	$\widehat{y} \in \widehat{A}^{\widehat{\phi}}$, 
	where $\widehat{y}$ 
	denote the image of $Y$ in $\widehat{A}$. Moreover, if $x_1,\ldots,x_m \in A^{\phi}$, then $\widehat{x_1},\ldots,\widehat{x_m} \in \widehat{A}^{\widehat{\phi}}$, where $\widehat{x_1},\ldots,\widehat{x_m}$ denote the images of $X_1,\ldots,X_m$ in $\widehat{A}$. 
\end{lem}

We shall now prove Theorem C which gives a necessary and sufficient condition for the ring $A$ in \eqref{AG} to be a polynomial ring, when $H$ is of the form 
\begin{equation}\label{h}
H(X_{1}, \ldots, X_{m}, Z,T)= h(Z,T)+ (X_{1} \cdots X_{m})g,
\end{equation}
for some $g \in k[X_1,\ldots,X_m,Z,T]$. We first prove a necessary criterion.

\begin{lem}\thlabel{rk1}
	
	Let $A$ be the affine domain as in \eqref{AG} with $H$ as in \eqref{h}. If $A=k^{[m+2]}$, then there exists an exponential map $\phi$ such that $k[x_1,\ldots, x_m] \subseteq A^{\phi} \not \subseteq k[x_1,\ldots, x_m,z,t]$.
\end{lem}

\begin{proof}
	By \thref{equiv}, $k[Z,T]=k[h,h_1]=k[h]^{[1]}$ for some $h_1 \in k[Z,T]$.
	Hence from \eqref{h}, it follows that 
	$$
	H=(X_1\cdots X_m) \widetilde{g}(X_1,\ldots, X_m,h,h_1)+h
	$$
	for some $\widetilde{g} \in k^{[m+2]}$. Now for $G=X_1^{r_1} \cdots X_m^{r_m}Y -H$, by \thref{rs} we have
	$$
	k[X_1,\ldots,X_m,Y,Z,T]=k[X_1,\ldots,X_m,Y,h,h_1]=k[X_1,\ldots,X_m,G,h_1,h_2]
	$$
	for some $h_2 \in k[X_1,\ldots,X_m,Y,Z,T]$. Hence 
	$A=k[x_1,\ldots,x_m,h_1(z,t), \overline{h_2}]=k[x_1,\ldots,x_m]^{[2]}$ where $\overline{h_2}$ denote the image of $h_2$ in $A$. Since $k[x_1,\ldots,x_m,z,t] \subsetneq A$ it follows that $\overline{h_2} \in A \setminus k[x_1,\ldots,x_m,z,t]$. We now consider the exponential map $\phi : A \rightarrow A[W]$ such that 
	$$
	\phi(x_i)=x_i,\text{~for every~}i, 1\leqslant i \leqslant m,~~\phi(h_1)=h_1+W,~~ \phi(\overline{h_2})=\overline{h_2}.
	$$
	Therefore, $A^{\phi}=k[x_1,\ldots,x_m,\overline{h_2}]$ and hence the assertion follows.
\end{proof}

We now prove Theorem C. 
\begin{thm}\thlabel{thpoly}
	Let $A$ be the affine domain as in \eqref{AG}, with $H$ as in \eqref{h}. Then the following statements are equivalent:
	\begin{itemize}
		\item [\rm (i)] $A$ is geometrically factorial over $k$ and there exists an exponential map $\phi$ on $A$ satisfying 
		$k[x_1,\ldots, x_m] \subseteq A^{\phi} \not \subseteq k[x_1,\ldots, x_m,z,t]$.
		
		\item [\rm (ii)] $A=k^{[m+2]}$.
	\end{itemize}
\end{thm}

\begin{proof}
    
    $\rm (i) \Rightarrow (ii):$
	By \thref{fdk1}, $\phi$ induces a non-trivial exponential map $\widehat{\phi}$ on 
	$$
	\widehat{A} \cong \frac{k[X_{1}, \ldots , \ldots, X_{m}, Y,Z,T]}{(X_{1}^{r_{1}} \cdots X_{m}^{r_{m}}Y-h( Z,T))},
	$$
	such that $\widehat{x_{1}},\ldots,\widehat{x_{m}},\widehat{y} \in \widehat{A}^{\widehat{\phi}}$, 
	where $\widehat{x_{1}},\ldots,\widehat{x_{m}},\widehat{y}$ denote the images of $X_1,\ldots,X_m,Y$ in $\widehat{A}$, respectively. 
 	We now show that $k[Z,T]=k[h]^{[1]}$.
	
	\medskip
	\noindent
	{\it Case} 1: Let $k$ be an infinite field. 
	Since $\widehat{y} \in \widehat{A}^{\widehat{\phi}}$, by \thref{subdk} it follows that $\dk(\widehat{A})=\widehat{A}$. Hence by \thref{dk}, we can assume that 
	$$
	h(Z,T)=a_0(Z)+a_1(Z)T
	$$
	 for some $a_0,a_1 \in k^{[1]}$. Let $\overline{k}$ be an algebraic closure of the field $k$.
	 As $A$ is geometrically factorial $h(Z,T)$ is irreducible in $\overline{k}[Z,T]$ (cf. \thref{ufd2}). 
	
	If $a_1(Z)=0$, then $a_0(Z) (=h(Z,T))$ is irreducible in $\overline{k}[Z,T]$, hence it is linear in $Z$. Thus $k[Z,T]=k[h]^{[1]}$.
	
	If $a_1(Z) \neq 0$, then $\gcd(a_0(Z),a_1(Z))=1$. Now $\widehat{\phi}$ induces a non-trivial exponential map on 
	$$
	\widetilde{A}=\widehat{A} \otimes_{k[\widehat{x_{1}},\ldots,\widehat{x_{m}},\widehat{y}]} k(\widehat{x_{1}},\ldots, \widehat{x_{m}},\widehat{y}) \cong \frac{L[Z,T]}{(\mu-h(Z,T))}=\frac{L[Z,T]}{(\mu -a_0(Z)-a_1(Z)T)},
	$$ 
	where $L=k(X_1,\ldots,X_m,Y)$, $\mu \in L \setminus k$, and hence $\gcd(\mu-a_0,a_1)=1$ in $L^{[1]}$. Since $\widetilde{A}$ is not rigid, $a_1 \in k^{*}$, and hence $k[Z,T]=k[h]^{[1]}$.
	
	\medskip
	\noindent
	{\it Case} 2: Let $k$ be a finite field.
	Now $\widehat{\phi}$ induces a non-trivial exponential map $\overline{\phi}$ on  
	$$
	\overline{A}:=\widehat{A} \otimes_k \overline{k} \cong \frac{\overline{k}[X_{1}, \ldots , \ldots, X_{m}, Y,Z,T]}{(X_{1}^{r_{1}} \cdots X_{m}^{r_{m}}Y-h( Z,T))},
	$$
	such that $\overline{x_{1}},\ldots,\overline{x_{m}},\overline{y} \in \overline{A}^{\overline{\phi}}$,
	where $\overline{x_{1}},\ldots,\overline{x_{m}},\overline{y}$ denote the images of $X_1,\ldots,X_m,Y$ in $\overline{A}$, respectively. 
	By Case 1, we have $\overline{k}[Z,T]=\overline{k}[h]^{[1]}$, and hence by \thref{sepco}, $k[Z,T]=k[h]^{[1]}$.
	
	Now from \thref{rs}, it follows that $A=k^{[m+2]}$.
	
	\smallskip
	\noindent
	$\rm (ii) \Rightarrow (i):$ Since $A=k^{[m+2]}$, it is geometrically factorial over $k$ and the rest follows from \thref{rk1}.
\end{proof}

\begin{rem}\thlabel{r2}
\em{ Note that the proof shows that the condition ``$A$ is geometrically factorial over $k$" may be replaced by a more specific condition that ``$A \otimes_k \overline{k}$ is a UFD" where $\overline{k}$ is an algebraic closure of $k$. However,
	the following example shows that in statement (i) of \thref{thpoly}, the condition ``$A$ is geometrically factorial" can not be relaxed to ``$A$ is a UFD".}
	
\end{rem}

\begin{ex}\thlabel{ex2}
	\em{ Let 
		$$
		A=\frac{\mathbb{R}[X_1,X_2,Y,Z,T]}{(X_1^2X_2^2Y-1-Z^2)},
		$$
		where $x_1,x_2,y,z,t$ denote the images of $X_1,X_2,Y,Z,T$ respectively in A.
		Note that $A=B[t]=B^{[1]}$, where $B=\frac{\mathbb{R}[X_1,X_2,Y,Z]}{(X_1^2X_2^2Y-1-Z^2)}$.
		
		We consider the exponential map $\phi: A \rightarrow A[W]$, such that $\phi|_B=id_B$ and $\phi(t)=t+W$. Then it follows that $A^{\phi}=B$. Therefore, $\mathbb{R}[x_1,x_2] \subseteq A^{\phi} \nsubseteq \mathbb{R}[x_1,x_2,z,t]$, as $y \in A^{\phi}$.
		
		Now note that here $H=h=1+Z^2$, which is an irreducible polynomial in $\mathbb{R}[Z,T]$ but not irreducible $\mathbb{C}[Z,T]$. 
		Hence by \thref{ufd2}, $A$ is a UFD but $A \otimes_{\mathbb{R}} \mathbb{C}$ is not a UFD. Therefore, $A \neq \mathbb{R}^{[4]}$.

}
\end{ex}

We now deduce the structure of $\dk(A)$ and $\ml(A)$ for the special form of $H$ as in \eqref{h}.

\begin{cor}\thlabel{ml3}
	Let $A$ be the affine domain as in \eqref{AG} with $H$ as in \eqref{h}. Then the following hold:
	\begin{itemize}
		\item [\rm(a)] If $h(Z,T) \in k^{*}$, then $\dk(A)=A$ and $\ml(A)=k[x_1,\ldots,x_m,x_1^{-1},\ldots,x_m^{-1}]$.
		\item[\rm(b)] If $h(Z,T) \notin k^{*}$,
		$\dk(A) \neq k[x_{1}, \ldots , x_{m},z,t]$ and $A$ is geometrically factorial , then $\ml(A) \subsetneq k[x_{1},\ldots,x_{m}]$. Moreover, if $m=1$ then $\ml(A)=k$.
	\end{itemize}
\end{cor}

\begin{proof}

	(a) If $h(Z,T) \in k^{*}$, then $A=k[x_1,\ldots,x_m,x_1^{-1},\ldots,x_m^{-1},z,t]$. Therefore, 
	$$
	\dk(A)=A \text{~~and~~} \ml(A)=k[x_1,\ldots,x_m,x_1^{-1},\ldots,x_m^{-1}].
	$$ 
	
	\smallskip
	\noindent
	(b) Since $h(Z,T) \notin k^{*}$, for every $i$, $1\leqslant i \leqslant m$, $x_i \notin A^{*}$. Therefore, by \thref{ml}(a), $\ml(A) \subseteq k[x_1,\ldots,x_m]$. 
	Since $\dk(A) \neq k[x_1,\ldots,x_m,z,t]$, there exists a non-trivial exponential map $\phi$ on $A$ such that $A^{\phi} \nsubseteq k[x_1,\ldots,x_m,z,t]$. Suppose, if possible, 
	$\ml(A)=k[x_1,\ldots,x_m]$. Then, $k[x_1,\ldots,x_m] \subseteq A^{\phi} \nsubseteq k[x_1,\ldots,x_m,z,t]$. But then, by \thref{thpoly}, 
	$A=k^{[m+2]}$, which is contradiction as $\ml(k^{[m+2]})=k$. 
	
	 Therefore, $\ml(A) \subsetneq k[x_1,\ldots,x_m]$, and since $\ml(A)$ is algebraically closed in $A$ (cf. \thref{prop}(i)), for $m=1$, $\ml(A)=k$.
\end{proof}

We now answer Question 2 for some special form of $H$. 
\begin{prop}\thlabel{ml1}
Let $C:=k[X_1,\ldots,X_m]$ and $A$ be the affine domain as in \eqref{AG} where 
	$$
	H(X_{1}, \ldots , X_{m},Z,T)=a_{0}(Z)+a_{1}(Z)T+\widetilde{H}(X_{1}, \ldots , X_{m},Z)
	$$
	and $\widetilde{H} \in (X_1,\ldots,X_m)C[Z]$. 
	Then  $\dk(A)=A$ and $\ml(A)=k$ when any one of following holds:
	\begin{enumerate}
		\item[\rm(i)] $a_1(Z)\neq 0$.
		\item[\rm(ii)] $a_1(Z)=0$, $H$ is a monic polynomial in $Z$ and $\left(H,H_Z\right)=C[Z]$.
	\end{enumerate}
\end{prop}
\begin{proof}	
	(i) It is clear that for every $i \in \{1,\ldots,m\}$, $x_i \notin A^{*}$. Let $Q \in k^{[m+2]}$ be such that 
	$$
	Q(x_{1}, \ldots , x_{m},y,z)=x_{1}^{r_{1}}\cdots x_{m}^{r_{m}}y-\widetilde{H}(x_{1}, \ldots , x_{m},z)-a_0(z).
	$$
	Note that $Q(x_1,\ldots,x_m,y,z)-a_1(z)t=0$.
	As $a_1(Z) \neq 0$, for every $j \in \{ 1, \ldots, m\}$, we now define the following maps $\phi_{j}: A \rightarrow A[U]$ by 
	$$
	\phi_{j}(x_{i})= x_{i} \text{~for~}i, 1 \leqslant i \leqslant m, i \neq j,\,\,
	\phi_{j}(x_{j})=x_{j}+ a_{1}(z)U,\,\,
	\phi_{j}(y)=y,\,\,
	\phi_{j}(z)=z,
	$$
	and
	$$
	\phi_{j}(t)= \frac{Q(x_1,\ldots,x_{j-1},x_j+a_1(z)U,x_{j+1},\ldots,x_m,y,z)}{a_1(z)}=t+ Uv_j(x_1,\ldots,x_m,y,z,U),
	$$
	for some $v_j \in k[x_1,\ldots,x_m,y,z,U]$. It is easy to see that $\phi_{j}\in \text{EXP}(A)$, for every $j$.
	Since $y \in A^{\phi_{j}}$, by \thref{subdk}, we have $\dk(A)=A$. 
	
	Let $C_j:=k[x_{1}, \ldots, x_{j-1},x_{j+1},\ldots, x_{m}, y, z ]$. Then 
	$C_j\subseteq A^{\phi_{j}} \subseteq A$. 
	As $C_j$ is algebraically closed in $A$ and $\td_{k} C_j= \td_{k} A^{\phi_{j}}=m+1$ (cf. \thref{prop}(ii)), 
	we have $A^{\phi_{j}}=C_j$. By \thref{ml}(a), $\ml(A) \subseteq k[x_{1},\ldots,x_{m}]$. 
	Since $j$ is arbitrarily chosen from $\{ 1,\ldots, m\}$, we get that 
	$\ml(A) \subseteq k[x_{1},\ldots,x_{m}]\bigcap_{1 \leqslant j \leqslant m} A^{\phi_{j}}=k$. 
	Thus $\ml(A)=k$.
	
	\smallskip
	\noindent
	(ii) As $a_1(Z)=0$, $A=B(r_1,\ldots,r_m,H)^{[1]}$. Now by \thref{stiso} we have $A \cong B(1,\ldots,1,H)^{[1]}$. Since by \thref{mld}, $\ml(B(1,\ldots,1,H))=k$, $\ml(A)=k$. As $B(1,\ldots,1,H)$ is not rigid, by \thref{rdk}, $\dk(A)=A$. 
\end{proof}

As a consequence, we give the complete description of $\dk(A)$ and $\ml(A)$ when $A$ (as in \eqref{AG}) is a regular domain over an infinite field and $H=h(Z,T)$.

\begin{cor}\thlabel{ml2}
Let $A$ be a regular domain defined by
$$
A = \frac{k[X_{1},\ldots, X_{m},Y,Z,T]}
{\left( X_{1}^{r_{1}}\cdots X_{m}^{r_{m}}Y-h(Z,T)\right)}.
$$
Then the following hold:
\begin{itemize}
\item[\rm(a)] If $h(Z,T)$ is coordinate in $k[Z,T]$, then $\dk(A)=A$ and $\ml(A)=k$.
\item[\rm(b)] If there exists a system of coordinates $\{Z_{1}, T_{1}\}$ of $k[Z,T]$ such that $h(Z,T)=a_{0}(Z_{1})+a_{1}(Z_{1})T_{1}$, then $\dk(A)=A$ and $\ml(A)=k$.
\item[\rm(c)] If $k$ is an infinite field and if $h(Z,T)$ can not be expressed in the form described in part (b) above, then $\dk(A)=k[x_{1},\ldots, x_{m},z,t]$ and $\ml(A)=k[x_{1}, \ldots , x_{m}]$.
	\end{itemize}

\end{cor}

\begin{proof}
	\noindent
	(a) As $k[Z,T]=k[h]^{[1]}$, $A=k^{[m+2]}$. Therefore, $\dk(A)=A$ and $\ml(A)=k$. 
	
	\medskip
	\noindent
	(b) If $a_1(Z_1) \neq 0$, then the assertion follows from \thref{ml1}(i).
	
	If $a_1(Z_1)=0$, then $h(Z,T)=a_0(Z_1)$. As $A$ is regular we have $\gcd(a_0,(a_0)_{Z_1})=1$. Hence the result follows from \thref{ml1}(ii).

	\medskip
	\noindent
	(c) A special case of \thref{r3}.
\end{proof}	

\section*{Acknowledgements}
The authors thank Professor Amartya K. Dutta for asking Question 2, suggesting the current version of \thref{thpoly} and carefully going through the draft. The first author acknowledges the Council of Scientific and Industrial Research (CSIR) for the Shyama Prasad Mukherjee fellowship (SPM-07/0093(0295)/2019-EMR-I). The second author also acknowledges Department of Science and Technology (DST), India for their INDO-RUSS project (DST/INT/RUS/RSF/P-48/2021).


\begin{thebibliography}{XXX}
	
	\bibitem{aeh} S.S. Abhyankar, P. Eakin and W. Heinzer, {\it On the uniqueness of the coefficient ring in a polynomial ring}, J. Algebra {\bf 23} (1972), 310--342.
	
	\bibitem{asa87} T. Asanuma, {\it Polynomial fibre rings of algebras over Noetherian rings}, Invent. Math. {\bf 87} (1987),
	101--127.	
	

	\bibitem{cra} A. J. Crachiola, {\it The hypersurface $x + x^2y + z^2 + t^3 = 0$ over a field of arbitrary characteristic},
	Proc. Amer. Math. Soc. {\bf{134}} (2005), 1289--1298.
	
	\bibitem{cra1} A. J. Crachiola, {\it On automorphisms of Danielewski surfaces} J. Algebraic Geom. {\bf 15(1)} (2006), no. 1, 111--132.
	
	\bibitem{dan} W. Danielewski, {\it On a cancellation problem and automorphism groups of affine algebraic varieties},
	Preprint, Warsaw, 1989.
	
	\bibitem{dom} H. Derksen, O. Hadas and L. Makar-Limanov, {\it Newton polytopes of invariants of additive group
		actions}, J. Pure Appl. Algebra {\bf{156}} (2001), 187--197.
	
	\bibitem{dub} A. Dubouloz, { \it 
		Additive group actions on Danielewski varieties and the cancellation problem},
	Math. Z. {\bf {255}(1)} (2007), 77--93.
	
	
	\bibitem{dutta} A.K. Dutta, { \it On separable $\mathbb{A}^{1}$-forms}, Nagoya Math. J. {\bf 159} (2000), 45--51.
	
	\bibitem{eh}  P. Eakin and W. Heinzer {\it A cancellation problem for rings}, Conference on Commutative Algebra, Lecture Notes in Mathematics  Vol. 311, pp. 61--77,  Springer, Berlin, 1973. 
	
	\bibitem{Guff} S.A. Gaifullin, { \it Automorphisms of Danielewski varieties}, J. Algebra {\bf {573}} (2021), 364--392. 
	
\bibitem{asa} P. Ghosh and N. Gupta, {\it On the triviality of a family of linear hyperplanes}, preprint, arXiv:2206.15210.

  	\bibitem{inv} N. Gupta, { \it On the cancellation problem for the affine space $\mathbb{A}^3$ in characteristic $p$}, Invent. Math. {\bf 195} (2014), 279--288.
	
	\bibitem{com} N. Gupta, {\it On the family of affine threefolds $x^{m}y=F(x,z,t)$ }, Compositio Math. {\bf 150} (2014), 979--998.
	
	\bibitem{adv} N. Gupta, 
	{\it On Zariski's Cancellation Problem in positive characteristic}, 
	Adv. Math. {\bf {264}} (2014), 296--307.
	
	\bibitem{Gicm} N. Gupta, {\it The Zariski Cancellation Problem and related problems
in Affine Algebraic Geometry}, to appear in ICM 2022 proceedings.

	\bibitem{miya} M. Miyanishi, {\it Lectures on Curves on rational and unirational surfaces}, Narosa Publishing House, New Delhi, 1978.
	

\end{thebibliography}
\end{document}